\documentclass[12pt,a4paper]{article}
\usepackage{amsmath,amssymb,amsthm,bm}
\usepackage{graphicx}
\usepackage{txfonts}

\newtheorem{theorem}{Theorem}[section]
\newtheorem{proposition}[theorem]{Proposition}
\newtheorem{corollary}[theorem]{Corollary}
\newtheorem{lemma}[theorem]{Lemma}
\newtheorem{remark}[theorem]{Remark}
\newtheorem{example}[theorem]{Example}

\DeclareMathOperator{\diam}{diam\,}
\newcommand{\1}{{\bm 1}}

\numberwithin{equation}{section}

\begin{document}

\begin{center}
\large\bf
A Classification of Graphs Through \\
Quadratic Embedding Constants and Clique Graph Insights
\end{center}

\bigskip

\begin{center}
Edy Tri Baskoro\\
Combinatorial Mathematics Research Group \\
Faculty of Mathematics and Natural Sciences \\
Institut Teknologi Bandung
Jalan Ganesa 10 Bandung, Indonesia \\
ebaskoro@math.itb.ac.id

\bigskip

Nobuaki Obata\\
Center for Data-driven Science and Artificial Intelligence \\
Tohoku University\\
Sendai 980-8576 Japan \\
and \\
Combinatorial Mathematics Research Group \\
Faculty of Mathematics and Natural Sciences \\
Institut Teknologi Bandung
Jalan Ganesa 10 Bandung, Indonesia \\
obata@tohoku.ac.jp

\end{center}

\bigskip

\begin{quote}
\textbf{Abstract}\enspace
The quadratic embedding constant (QEC)
of a graph $G$ is a new numeric invariant,
which is defined in terms of the distance matrix
and is denoted by $\mathrm{QEC}(G)$.
By observing graph structure of the maximal cliques
(clique graph),
we show that a graph $G$ with $\mathrm{QEC}(G)<-1/2$
admits a ``cactus-like'' structure.
We derive a formula for the quadratic embedding constant
of a graph consisting of two maximal cliques.
As an application we discuss characterization of
graphs along the increasing sequence of $\mathrm{QEC}(P_d)$,
where $P_d$ is the path on $d$ vertices.
In particular, we determine graphs $G$ satisfying
$\mathrm{QEC}(G)<\mathrm{QEC}(P_5)$.
\end{quote}

\begin{quote}
\end{quote}

\begin{quote}
\end{quote}

\section{Introduction}

In the recent paper \cite{Obata-Zakiyyah2018} 
the \textit{quadratic embedding constant} 
(\textit{QE constant} for short)
of a finite connected graph $G=(V,E)$ with $|V|\ge2$ 
is defined by
\begin{equation}\label{0eqn:def od QEC(G)}
\mathrm{QEC}(G)
=\max\{\langle f,Df \rangle\,;\, f\in C(V), \,
\langle f,f \rangle=1, \, \langle \1,f \rangle=0\},
\end{equation}
where $C(V)$ is the space of all $\mathbb{R}$-valued functions
on $V$, $\bm{1}\in C(V)$ the constant function
taking value 1, and $\langle\cdot,\cdot\rangle$ the 
canonical inner product.
The QE constant is profoundly related to
the quadratic embedding of a graph in a Euclidean space
\cite{Jaklic-Modic2013, Jaklic-Modic2014, 
Schoenberg1935,Schoenberg1938}
or more generally to Euclidean distance geometry
\cite{Alfakih2018, Balaji-Bapat2007,
Deza-Laurent1997,Liberti-Lavor-Maculan-Mucherino2014}.
In fact, it is essential to note that
a graph $G$ admits a quadratic embedding in Euclidean space
if and only if $\mathrm{QEC}(G)\le0$.
In recent years, the QE constant has garnered growing interest
as a new numeric invariant of graphs.
The QE constants of graphs of particular classes
are known explicitly,
see \cite{Choudhury-Nandi2023,Irawan-Sugeng2021,
Mlotkowski2022,Obata2017,
Obata2023a,Obata2023b,Obata-Zakiyyah2018,Purwaningsih-Sugeng2021},
and the formulas in relation to graph operations are 
established in \cite{Lou-Obata-Huang2022, MO-2018}.
Moreover, a table of the QE constants of graphs
on $n\le5$ vertices is available \cite{Obata-Zakiyyah2018},
where the value of the graph No.~12 on $n=5$ vertices is wrong
and the correction is found in \cite{Baskoro-Obata2021}.

With the above mentioned background,
we are naturally led to the forward-thinking project
of classifying graphs by means of the QE constants.
In \cite{Baskoro-Obata2021} we initiated an attempt
to classify graphs along with $\mathrm{QEC}(P_d)$, 
the QE constant of the path $P_d$ on $d\ge2$ vertices,
which forms an increasing sequence as
\begin{equation}\label{01eqn:scale}
-1=\mathrm{QEC}(P_2)<
\mathrm{QEC}(P_3)<
\dotsb
<\mathrm{QEC}(P_d)<
\dotsb \rightarrow -\frac12\,.
\end{equation}
In fact, it is known \cite{Mlotkowski2022} that
\begin{equation}\label{01eqn:QEC(P_d)}
\mathrm{QEC}(P_d)
=-\left(1+\cos\frac{\pi}{d}\right)^{-1},
\qquad d\ge2.
\end{equation}
In this paper, we study graphs $G=(V,E)$ with
$\mathrm{QEC}(G)<-1/2$ by means of 
the \textit{clique graph} $\Gamma(G)$.
Here the clique graph is 
a graph $\Gamma(G)=(\mathcal{V},\mathcal{E})$,
where $\mathcal{V}$ is the set of maximal cliques of $G$
and $\{H_1,H_2\}\in\mathcal{E}$ if and only if
$H_1\neq H_2$ and $H_1\cap H_2\neq\emptyset$.
A key result is that the clique graph of 
a graph $G$ with $\mathrm{QEC}(G)<-1/2$ is a tree.
Then, combining the result on forbidden subgraphs
\cite{Baskoro-Obata2021}, we conclude that
a graph $G$ with $\mathrm{QEC}(G)<-1/2$ consists of
maximal cliques which form a ``cactus-like'' structure,
for the precise statement see 
Theorem \ref{04thm:necessary condition}.
As an application we discuss graphs $G$ satisfying
$\mathrm{QEC}(G)<\mathrm{QEC}(P_5)$.

It is noteworthy that the QE constant provides
additional information to the distance spectra
and raises interesting questions,
for the distance spectra see e.g.,
\cite{Aouchiche-Hansen2014,Balaji-Bapat2007,
Indulal-Gutman2008,Lin-Hong-Wang-Shub2013}.
In fact, it is known \cite{Lou-Obata-Huang2022}
that $\delta_2(G)\le \mathrm{QEC}(G)
<\delta_1(G)$, where $\delta_1(G)$ and $\delta_2(G)$ are the
largest and the second largest eigenvalues of the
distance matrix of $G$, respectively.
It is straightforward to see that
$\delta_2(G)=\mathrm{QEC}(G)$ holds if the distance matrix
of $G$ has a constant row sum (in some literatures,
such a graph is called \textit{transmission regular}).
But the converse is not true as the paths $P_n$ with even $n$
are counter-examples \cite{Mlotkowski2022}.
In this aspect characterization of graphs satisfying 
$\delta_2(G)=\mathrm{QEC}(G)$ is an interesting question.
On the other hand, 
the second largest eigenvalue $\delta_2(G)$ has
been adopted for classifying graphs,
in particular, the distance-regular graphs $G$ with $\delta_2(G)\le0$ are
classified \cite{Koolen-Shpectorov1994}.
In the recent paper \cite{Guo-Zhou2023}
the bicyclic graphs and the split graphs $G$ with $\delta_2(G)<-1/2$ are
characterized.
A detailed comparison is naturally expected to be very interesting
and will appear elsewhere.

The paper is organized as follows.
In Section 2 we assemble some basic notations for the QE constants.
In Section 3 we examine some properties of 
the clique graph $\Gamma(G)$ and show 
a relation between the diameters of $G$ and $\Gamma(G)$ 
(Propositions \ref{03prop:diam(CG)}
and \ref{03prop:Clique graph is tree}). 
In Section 4 we prove 
the main result (Theorem \ref{04thm:necessary condition}).
In Section 5 we determine a graph with exactly two
maximal cliques.
In Section 6 we discuss graphs $G$ satisfying
$\mathrm{QEC}(G)<\mathrm{QEC}(P_5)$.
In Appendix we derive a formula for the QE constant of
a graph with exactly two maximal cliques.

\section{Quadratic Embedding Constants}

Throughout the paper a graph $G=(V,E)$ is a pair,
where $V$ is a non-empty finite set and $E$ is a set of
two-element subsets $\{x,y\}\subset V$.
As usual, elements of $V$ and $E$ are respectively called
a \textit{vertex} and an \textit{edge}. 
Two vertices $x,y\in V$ are called \textit{adjacent}
if $\{x,y\}\in E$, and we also write $x\sim y$.
In that case we have $x\neq y$ necessarily.

For $s\ge0$ a sequence of vertices $x=x_0,x_1,\dots, x_s=y\in V$
is called a \textit{walk} connecting $x$ and $y$ of
length $s$ if they are successively adjacent:
\begin{equation}\label{02eqn:walk}
x=x_0\sim x_1 \sim \dotsb \sim x_s=y.
\end{equation}
A graph $G$ is called \textit{connected}
if any pair of vertices are connected by a walk.
In this paper a graph is always assumed to be connected
unless otherwise stated.
The length of a shortest walk connecting two vertices
$x,y\in V$ is called the \textit{graph distance} between
$x$ and $y$ in $G$ and is denoted by $d_G(x,y)$. 
A walk in \eqref{02eqn:walk} is called
a \textit{shortest path} connecting $x$ and $y$
if $s=d_G(x,y)$ holds.
In that case $x_0,x_1,\dots, x_s$ are
mutually distinct.

As is defined in \eqref{0eqn:def od QEC(G)},
the \textit{QE constant} of a graph $G=(V,E)$,
denoted by $\mathrm{QEC}(G)$, is the conditional maximum of
the quadratic function $\langle f, Df\rangle$
associated to the distance matrix $D=[d_G(x,y)]$
subject to two constraints $\langle f, f\rangle=1$ and
$\langle \bm{1}, f\rangle=0$.

In this section we assemble some basic results,
for more details 
see e.g., \cite{Baskoro-Obata2021,MO-2018,Obata-Zakiyyah2018}.
We first recall a computational formula
based on the standard method of Lagrange's multipliers.

\begin{proposition}[\cite{Obata-Zakiyyah2018}]
\label{02prop:QEC}
Let $D$ be the distance matrix of a graph $G=(V,E)$ 
on $n=|V|$ vertices with $n\ge3$.
Let $\mathcal{S}$ be the set of all stationary points  
$(f,\lambda,\mu)$ of
\begin{equation}\label{02eqn:basic Lagrange}
\varphi(f,\lambda,\mu)
=\langle f,Df\rangle
 -\lambda(\langle f,f\rangle-1)
 -\mu\langle \1,f\rangle,
\end{equation}
where $f\in C(V)\cong \mathbb{R}^n$, $\lambda\in\mathbb{R}$
and $\mu\in\mathbb{R}$.
Then we have
\[
\mathrm{QEC}(G)
=\max\{\lambda\,;\, (f,\lambda,\mu)\in \mathcal{S}\}.
\]
\end{proposition}

In general, a graph $H=(V^\prime, E^\prime)$ is called a
\textit{subgraph} of $G=(V,E)$
if $V^\prime\subset V$ and $E^\prime\subset E$.
If both $G$ and $H$ are connected,
they have their own graph distances.
If they coincide in such a way that
\[
d_H(x,y)=d_G(x,y),
\qquad x,y\in V^\prime,
\]
we say that $H$ is \textit{isometrically embedded} in $G$.
The next assertions are immediate from definition but useful.

\begin{proposition}[\cite{Obata2023a, Obata2023b}]
\label{02prop:isometric embedding}
Let $G=(V,E)$ be a graph and $H=(V^\prime,E^\prime)$ a subgraph.
\begin{enumerate}
\setlength{\itemsep}{0pt}
\item[\upshape (1)] If $H$ is isometrically embedded in $G$,
then $H$ is an induced subgraph of $G$.
\item[\upshape (2)] If $H$ is an induced subgraph of $G$ and
\[
\mathrm{diam\,}(H)=\max\{d_H(x,y)\,;\, x,y\in V^\prime\}\le 2,
\]
then $H$ is isometrically embedded in $G$.
\end{enumerate}
\end{proposition}

\begin{proposition}[\cite{Obata-Zakiyyah2018}]
\label{02prop:isometrically embedded subgraphs}
Let $G=(V,E)$ and $H=(V^\prime,E^\prime)$ be two graphs
with $|V|\ge2$ and $|V^\prime|\ge2$.
If $H$ is isometrically embedded in $G$, we have
\begin{equation}\label{02eqn:QEC(H)le QEC(G)}
\mathrm{QEC}(H)\le \mathrm{QEC}(G).
\end{equation}
In particular, \eqref{02eqn:QEC(H)le QEC(G)} holds
if $H$ is an induced subgraph of $G$ and
$\mathrm{diam\,}(H)\le 2$.
\end{proposition}

Since any graph $G=(V,E)$ with $|V|\ge2$ contains at least
one edge, it has a subgraph $K_2$ isometrically embedded in $G$.
It then follows from Proposition
\ref{02prop:isometrically embedded subgraphs} that
\[
\mathrm{QEC}(G)\ge \mathrm{QEC}(K_2)=-1.
\]
Moreover, we have the following assertion,
see also Proposition \ref{06prop:QEC=-1}.

\begin{proposition}[\cite{Baskoro-Obata2021}]
\label{02prop:QEC ge -1}
For a graph $G$ we have $\mathrm{QEC}(G)=-1$ 
if and only if $G$ is a complete graph.
\end{proposition}

Thus, in order to determine $\mathrm{QEC}(G)$ of 
a graph $G$ which is not a complete graph,
it is sufficient to seek out the stationary points of
$\varphi(f,\lambda,\mu)$ with $\lambda>-1$ and 
then to specify the maximum of $\lambda$ appearing therein.

\section{Clique Graphs}

Most of this section follows a standard argument; 
however, to avoid ambiguity, we present some basic
properties of clique graphs.

Let $G=(V,E)$ be a graph (always assumed to be finite and connected).
For a non-empty subset $H\subset V$,
the subgraph induced by $H$ is denoted by $\langle H\rangle$.
By definition the vertex set of $\langle H\rangle$ is $H$ itself
and two-element subset $\{x,y\}\subset H$ belongs to the edge set 
of $\langle H\rangle$ if and only if $\{x,y\}\in E$.
A non-empty subset $H\subset V$ is called a \textit{clique} of $G$
if $\langle H \rangle$ is a complete graph.
A clique is called \textit{maximal} if
it is maximal in the family of cliques with respect to the
inclusion relation.
Except for notation, 
`clique' may also refer to the subgraph it induces.

Obviously, for a clique $H_0$ there exists a maximal clique
$H$ such that $H_0\subset H$.
In particular, for two vertices $a\sim b$ there 
exists a maximal clique containing $\{a,b\}$.

\begin{lemma}\label{03lem:two maximal cliques}
Let $H_1$ and $H_2$ be maximal cliques of a graph $G$
such that $H_1\neq H_2$.
Then $H_1\backslash H_2\neq\emptyset$ and
$H_2\backslash H_1\neq\emptyset$.
Moreover, there exist $a\in H_1\backslash H_2$
and $b\in H_2\backslash H_1$ such that $a\not\sim b$.
\end{lemma}

\begin{proof}
Suppose that $H_1\backslash H_2=\emptyset$ or
$H_2\backslash H_1=\emptyset$.
If the former occurs, we have $H_1\subset H_2$.
Since both $H_1$ and $H_2$ are maximal and
$H_1\neq H_2$ by assumption,
we come to a contradiction.

For the second half of the assertion,
suppose that any pair of $x\in H_1\backslash H_2$
and $y\in H_2\backslash H_1$ are adjacent.
We will show that $H_1\cup H_2$ is a clique.
In fact, take a pair of distinct vertices $x,y \in H_1\cup H_2$. 
If $x,y\in H_1$ or $x,y\in H_2$,
they are adjacent since $H_1$ and $H_2$ are cliques.
If otherwise, 
we have $x\in H_1\backslash H_2$ and $y\in H_2\backslash H_1$
or vice versa, and hence $x\sim y$ by assumption.
Consequently, for any pair of distinct 
vertices $x,y\in H_1\cup H_2$ we have $x\sim y$,
namely, $H_1\cup H_2$ becomes a clique.
Since $H_1\cup H_2$ contains $H_1$ and $H_2$ properly,
we come to a contradiction.
\end{proof}

For a graph $G=(V,E)$
let $\mathcal{V}$ be the set of all maximal cliques
and $\mathcal{E}$ the set of two-element subsets
$\{H_1,H_2\}\subset \mathcal{V}$ such that
$H_1\cap H_2 \neq\emptyset$.
Then $\Gamma(G)=(\mathcal{V},\mathcal{E})$ becomes a
(in fact, connected) graph,
which is called the \textit{clique graph} of $G$.
Accordingly, 
for two maximal cliques $H_1$ and $H_2$ of $G$ we write
$H_1\sim H_2$ if $H_1\neq H_2$ and $H_1\cap H_2\neq\emptyset$.
For more information on the clique graph,
see e.g., \cite{Roberts-Spencer1971,Szwarcfiter2003}.

\begin{lemma}\label{03lem:Gamma(G) is connected}
The clique graph $\Gamma(G)$ of a graph $G$ 
 (always assumed to be connected) is connected.
\end{lemma}

\begin{proof}
Let $H_1, H_2$ be two maximal cliques such that $H_1\neq H_2$.
By Lemma \ref{03lem:two maximal cliques} we may choose 
$a\in H_1\backslash H_2$ and $b\in H_2\backslash H_1$.
Since $G$ is connected,
there exists a walk connecting $a$ and $b$, say,
\[
a=x_0\sim x_1\sim \dotsb \sim x_s=b,
\]
where $s\ge1$.
For $1\le i \le s$ take a maximal clique $J_i$ 
containing $\{x_{i-1}, x_i\}$.
Then $x_i\in J_i\cap J_{i+1}$ implies that
$J_i=J_{i+1}$ or $J_i\sim J_{i+1}$.
Moreover, it follows from $a=x_0\in H_1\cap J_1$
that $H_1=J_1$ or $H_1\sim J_1$.
Similarly, $H_2=J_s$ or $H_2\sim J_s$.
In any case, $H_1$ and $H_2$ are connected by a walk
consisting of $J_1,J_2,\dots, J_s$.
\end{proof}

\begin{example}\normalfont
For a complete graph $K_n$ with $n\ge1$,
a path $P_n$ with $n\ge2$, and a cycle $C_n$ with $n\ge3$ we have
\[
\Gamma(K_n)=K_1,
\qquad
\Gamma(P_n)=P_{n-1},
\qquad
\Gamma(C_n)=C_n\,.
\]
If every maximal clique of a graph $G$ is $K_2$,
the clique graph $\Gamma(G)$ is nothing else but the
line graph of $G$.
Examples of this type are $G=P_n$ and $G=C_n$.
\end{example}

\begin{lemma}\label{03lem:constructing a chain of maximal cliques}
For $d\ge1$ let
\begin{equation}\label{03eqn:connecting x0 and xd}
x_0\sim x_1\sim \dotsb \sim x_d\,,
\end{equation}
be a shortest path connecting $x_0$ and $x_d$,
that is, $d(x_0,x_d)=d$.
For $1\le i\le d$ let $H_i$ be a maximal clique containing
$\{x_{i-1},x_i\}$.
Then,
\begin{equation}\label{03eqn:connecting H_1 and H_d}
H_1\sim H_2\sim \dotsb \sim H_d
\end{equation}
and $d(H_1,H_d)=d-1$.
Hence \eqref{03eqn:connecting H_1 and H_d} is a shortest path
connecting $H_1$ and $H_d$, 
and $H_1, H_2, \dots,H_d$ are mutually distinct.
\end{lemma}

\begin{proof}
For $1\le i\le d-1$ we have
$x_i\in H_i\cap H_{i+1}$ by definition,
and hence $H_i=H_{i+1}$ or $H_i\sim H_{i+1}$.
Suppose that $H_i=H_{i+1}$ occurs.
Then $x_{i-1},x_i, x_{i+1}\in H_i$ and these three vertices are
mutually distinct because 
\eqref{03eqn:connecting x0 and xd} is a shortest path.
Since $H_i$ is a clique, we have $x_{i-1}\sim x_{i+1}$,
which contradicts to that 
\eqref{03eqn:connecting x0 and xd} is a shortest path.
Thus we obtain a walk as in \eqref{03eqn:connecting H_1 and H_d}.

We next prove that \eqref{03eqn:connecting H_1 and H_d} gives
rise to a shortest path.
Let $s=d(H_1,H_d)$ and take a shortest path
\[
H_1=J_0\sim J_1\sim \dots \sim J_s=H_d\,.
\]
In that case we have
\begin{equation}\label{03eqn:in proof lemma 3.4}
s\le d-1.
\end{equation}
For $1\le i\le s$ we take $y_i\in J_{i-1}\cap J_i$.
Then $y_1\sim y_2\sim\dotsb\sim y_s$.
Moreover, since $x_0, y_1\in H_1=J_0$ we have
$x_0=y_1$ or $x_0\sim y_1$.
Similarly, $x_d=y_s$ or $x_d\sim y_s$.
Thus we obtain a walk connecting $x_0$ and $x_d$ whose length
is $s-1$, $s$ or $s+1$.
Hence $d=d(x_0,x_d)\le s+1$.
Combining \eqref{03eqn:in proof lemma 3.4} we obtain
$s=d-1$ and hence $d(H_1,H_d)=d-1$ as desired.
\end{proof}

\begin{proposition}\label{03prop:diam(CG)}
Let $G$ be a graph and $\Gamma(G)$ its clique graph.
Then
\begin{equation}\label{03eqn:diam(CG)}
\diam(G)-1\le \diam(\Gamma(G)).
\end{equation}
\end{proposition}

\begin{proof}
It is sufficient to show the assertion
for a graph $G$ with $d=\diam(G)\ge1$.
We take a shortest path
$x_0\sim x_1\sim \dotsb \sim x_d$ such that $d(x_0,x_d)=d$.
Define a sequence of maximal cliques $H_1,\dots, H_d$
as in Lemma \ref{03lem:constructing a chain of maximal cliques}.
Then we have
\[
d-1 =d(H_1,H_d)\le \diam (\Gamma(G)),
\]
which completes the proof of \eqref{03eqn:diam(CG)}.
\end{proof}

\begin{lemma}\label{03lem:constructing a chain of vertices}
For $d\ge1$ let 
\begin{equation}\label{03eqn:connecting H_0 and H_d}
H_0\sim H_1\sim \dotsb \sim H_d
\end{equation}
be a shortest path connecting $H_0$ and $H_d$,
that is $d(H_0,H_d)=d$.
For $1\le i\le d$ take a vertex $x_i\in H_{i-1}\cap H_i$.
Then
\begin{equation}\label{03eqn:path connecting x1 and xd}
x_1\sim x_2\sim \dotsb \sim x_d
\end{equation}
and $d(x_1,x_d)=d-1$.
Hence \eqref{03eqn:path connecting x1 and xd} is a shortest path
connecting $x_1$ and $x_d$,
and $x_1,x_2, \dots, x_d$ are mutually distinct. 
\end{lemma}

\begin{proof}
For $1\le i \le d-1$ we have $x_i, x_{i+1}\in H_i$
and hence $x_i=x_{i+1}$ or $x_i\sim x_{i+1}$.
Suppose that $x_i=x_{i+1}$ occurs.
Then $x_i=x_{i+1}\in H_{i-1}\cap H_i\cap H_{i+1}$
and hence $H_{i-1}\cap H_{i+1}\neq\emptyset$,
from which we obtain $H_{i-1}=H_{i+1}$ or $H_{i-1}\sim H_{i+1}$.
In any case we come to a contradiction because
\eqref{03eqn:connecting H_0 and H_d} is a shortest path.
We have thus obtained a walk
as in \eqref{03eqn:path connecting x1 and xd}.

We next prove that \eqref{03eqn:path connecting x1 and xd}
gives rise to a shortest path.
We set $s=d(x_1,x_d)$ and take a shortest path
connecting $x_1$ and $x_d$, say,
\[
x_1=y_0\sim y_1\sim \dots \sim y_s=x_d.
\]
In that case we have
\begin{equation}\label{03eqn:in proof lemma 3.6}
s\le d-1.
\end{equation}
For $1\le i \le s$ let $J_i$ be a maximal clique
containing $\{y_{i-1}, y_i\}$.
It then follows from
Lemma \ref{03lem:constructing a chain of maximal cliques} that
$d(J_1,J_s)=s-1$.
Since $x_1=y_0\in H_0\cap J_1$, we have 
$H_0=J_1$ or $H_0\sim J_1$.
Similarly, we have $H_d=J_s$ or $H_d\sim J_s$.
Thus, we obtain a walk connecting $H_0$ and $H_d$ of which length
is $s-1$, $s$ or $s+1$.
Hence $d=d(H_0,H_d)\le s+1$.
Combining \eqref{03eqn:in proof lemma 3.6}, we obtain
$s=d-1$ and hence $d(x_1,x_d)=d-1$ as desired.
\end{proof}

\begin{proposition}\label{03prop:Clique graph is tree}
Let $G$ be a graph and $\Gamma(G)$ its clique graph.
If $\Gamma(G)$ is a tree, we have
\begin{equation}\label{03eqn:diam(CG) for tree}
\diam(G)-1=\diam(\Gamma(G)).
\end{equation}
\end{proposition}

\begin{proof}
Set $d=\diam(\Gamma(G))$ and take a shortest path
\begin{equation}\label{03eqn:diameter path}
H_0\sim H_1\sim \dots \sim H_d,
\end{equation}
where $d(H_0,H_d)=d$.
For $1\le i \le d$ take $x_i\in H_{i-1}\cap H_i$.
By Lemma \ref{03lem:constructing a chain of vertices} we have
a shortest path
$x_1\sim x_2\sim \dotsb\sim x_d$.
Moreover, we take $x_0\in H_0\backslash H_1$ and
$x_{d+1}\in H_d\backslash H_{d-1}$.
Thus we obtain a walk
\begin{equation}\label{03eqn:diameter path x}
x_0\sim x_1\sim x_2\sim \dotsb\sim x_d\sim x_{d+1}
\end{equation}
whose length is $d+1$.

We shall prove that \eqref{03eqn:diameter path x} is a shortest path.
Set $s=d(x_0,x_{d+1})$ and take a shortest path, say,
\begin{equation}\label{03eqn:shortest path y}
x_0=y_0\sim y_1\sim y_2\sim \dotsb\sim y_s=x_{d+1}
\end{equation}
For $1\le i\le s$ let $J_i$ be 
a maximal clique containing $\{y_{i-1},y_i\}$.
By Lemma \ref{03eqn:connecting H_1 and H_d} we obtain 
a shortest path $J_1\sim \dotsb\sim J_s$,
namely,
\begin{equation}\label{03eqn:in proof prop 3.7}
d(J_1,J_s)=s-1.
\end{equation}
Now note that $x_0=y_0 \in H_0\cap J_1$.
Then we have $H_0=J_1$ or $H_0\sim J_1$.
Since $\Gamma(G)$ is a tree,
the ends of a diameter \eqref{03eqn:diameter path} are
pending vertices.
Hence $H_0\sim J_1$ implies that $J_1=H_1$.
In that case we have $x_0=y_0\in H_0\cap J_1=H_0\cap H_1$.
On the other hand,
we chose $x_0\in H_0\backslash H_1$, which is a contradiction.
Therefore, $H_0\sim J_1$ does not occur and we have $H_0=J_1$.
In a similar manner, we see that $H_d=J_s$.
Consequently, combining \eqref{03eqn:in proof prop 3.7} we come to
\[
d=d(H_0,H_d)=d(J_1,J_s)=s-1.
\]
Thus,
\[
\diam(\Gamma(G))=d=s-1=d(x_0,x_{d+1})-1\le \diam(G)-1.
\]
Finally, combining Proposition \ref{03prop:diam(CG)},
we obtain the equality \eqref{03eqn:diam(CG) for tree}.
\end{proof}

\section{Graphs with $\mathrm{QEC}(G)<-1/2$}

The complete bipartite graph $K_{1,3}$ is called a \textit{claw}.
The complete tripartite graph $K_{1,1,2}$,
which is also
obtained by deleting an edge from the complete graph $K_4$,
is called a \textit{diamond},
see Figure \ref{fig:Claw-Diamond}.
\begin{figure}[hbt]
\begin{center}
\includegraphics[width=240pt]{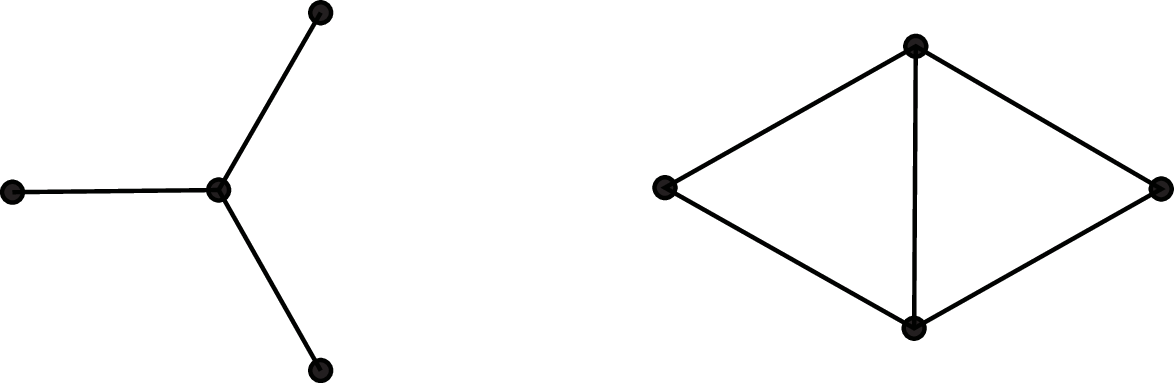}
\end{center}
\caption{Claw $K_{1,3}$ (left) and diamond $K_{1,1,2}$ (right)}
\label{fig:Claw-Diamond}
\end{figure}

It is essential to note that
\[
\mathrm{QEC}(K_{1,3})=\mathrm{QEC}(K_{1,1,2})=-\frac12\,.
\]
Since $\diam (K_{1,3})=\diam (K_{1,1,2})=2$,
we see from Propositions 
\ref{02prop:isometric embedding}
and \ref{02prop:isometrically embedded subgraphs} that
if a graph $G$ contains $K_{1,3}$ or $K_{1,1,2}$
as an induced subgraph,
we have $\mathrm{QEC}(G)\ge -1/2$.
Thus, we come to the following criterion.

\begin{proposition}[{\cite[Corollary 4.1]{Baskoro-Obata2021}}]
\label{03prop:forbidden subgraphs}
Any graph $G$ with $\mathrm{QEC}(G)< -1/2$ does not contain a
claw $K_{1,3}$ nor a diamond $K_{1,1,2}$ as an 
induced subgraph.
In short, the claw and diamond
are forbidden subgraphs for a graph with $\mathrm{QEC}(G)< -1/2$.
\end{proposition}

\begin{lemma}\label{04lem:two maximal cliques of G with <-1/2}
Let $G$ be a graph with $\mathrm{QEC}(G)< -1/2$.
If $H_1$ and $H_2$ are maximal cliques of $G$ with
$H_1\neq H_2$, then $H_1\cap H_2=\emptyset$ or $|H_1\cap H_2|=1$.
\end{lemma}

\begin{proof}
In order to prove the assertion by contradiction,
we suppose $|H_1\cap H_2|\ge2$ and take $x,y\in H_1\cap H_2$
with $x\neq y$.
By Lemma \ref{03lem:two maximal cliques} there exist
$a\in H_1\backslash H_2$ and 
$b\in H_2\backslash H_1$ such that $a\not\sim b$.
Then $\langle x,y,a,b\rangle$ forms a diamond,
which is a forbidden subgraph
as stated in Proposition \ref{03prop:forbidden subgraphs}.
\end{proof}

\begin{lemma}\label{03lem:three maximal cliques of G with <-1/2}
Let $G$ be a graph with $\mathrm{QEC}(G)< -1/2$.
If $H_1, H_2$ and $H_3$ are mutually distinct 
maximal cliques of $G$, 
then $H_1\cap H_2\cap H_3=\emptyset$.
\end{lemma}
 
\begin{proof}
In order to prove the assertion by contradiction
we suppose that $H_1\cap H_2\cap H_3\neq\emptyset$.
Then $H_1\cap H_2\neq\emptyset$ and 
by Lemma \ref{04lem:two maximal cliques of G with <-1/2} we have
$H_1\cap H_2=\{x\}$ for some $x\in V$.
Hence $H_2\cap H_3=H_3\cap H_1=H_1\cap H_2\cap H_3=\{x\}$.
On the other hand, by
Lemma \ref{03lem:two maximal cliques} 
there exist $a\in H_1\backslash H_2$
and $b\in H_2\backslash H_1$ such that $a\not\sim b$.
Note that there exists $c\in H_3\backslash H_1$ such that
$c\not\sim a$.
In fact, if any $c\in H_3\backslash H_1$ is adjacent to $a$,
then $H_3\cup\{a\}$ becomes a clique and we come to a contradiction.
Thus, we have chosen four vertices $x,a,b,c$.
There are two cases. 
In case of $b\not\sim c$, the induced subgraph
$\langle x,a,b,c\rangle$ becomes a claw.
In case of $b\sim c$ note that 
there exist $c^\prime\in H_3\backslash H_2$
such that $c^\prime\not\sim b$.
In fact,
if any $c^\prime\in H_3\backslash H_2$ is adjacent to
$b$, $H_3\cup\{b\}$ becomes a clique and we come to a contradiction.
Thus, taking $c^\prime\in H_3\backslash H_2$
such that $c^\prime\not\sim b$,
we see that $\langle x,b,c,c^\prime\rangle$ becomes a diamond.
In any case we obtain a forbidden subgraph 
as stated in Proposition \ref{03prop:forbidden subgraphs}
and arrive to a contradiction.
\end{proof}

\begin{proposition}\label{04prop:Gamma(G) is a tree}
For a graph $G=(V,E)$ with $\mathrm{QEC}(G)<-1/2$ the
clique graph $\Gamma(G)$ is a tree.
\end{proposition}

\begin{proof}
Suppose that the clique graph $\Gamma(G)$ is not a tree
and take a smallest cycle, say,
\begin{equation}\label{04eqn:smallest cycle in clique graph}
H_1\sim H_2\sim\dotsb \sim H_k\sim H_1,
\qquad k\ge3,
\end{equation}
where $H_i$ is a maximal clique of $G$ and
$H_i\cap H_{i+1}\neq\emptyset$ for $1\le i\le k$
(understanding $H_{k+1}=H_1$).
It follows from 
Lemma \ref{04lem:two maximal cliques of G with <-1/2} that
there exists a unique vertex $x_i$ such that
$H_i\cap H_{i+1}=\{x_i\}$ for $1\le i\le k$.
Then, obviously
\[
x_1\sim x_2\sim \dots \sim x_k\sim x_1.
\]

(Case 1) $k=3$.
In that case $x_1,x_2,x_3$ are mutually distinct.
We note that $x_1, x_3\in H_1$ and $x_2\not\in H_1$.
If $H_1=\{x_1,x_3\}$, namely $H_1\backslash\{x_1, x_3\}=\emptyset$,
then $H_1\cup \{x_2\}$ becomes a clique containing $H_1$ properly
and we come to a contradiction.
Hence $H_1\backslash\{x_1, x_3\}\neq\emptyset$.
If any $y\in H_1\backslash\{x_1, x_3\}$ is adjacent to $x_2$, 
then $H_1\cup \{x_2\}$ becomes a clique containing $H_1$ properly
and we come to a contradiction again.
Therefore, there exists $y\in H_1\backslash\{x_1, x_3\}$ such that
$y\not\sim x_2$.
Thus, $\langle x_1,x_2,x_3,y\rangle$ becomes a diamond,
which is a forbidden subgraph by Proposition
\ref{03prop:forbidden subgraphs}.
Consequently,
$\Gamma(G)$ does not contain a cycle
\eqref{04eqn:smallest cycle in clique graph} with $k=3$.

(Case 2) $k\ge4$.
Using the assumption 
that \eqref{04eqn:smallest cycle in clique graph} is a smallest
cycle, one can show easily that
$x_1,\dots, x_k$ are mutually distinct.

We first prove that the induced subgraph
$C=\langle x_1,x_2,\dots, x_k\rangle$ becomes a cycle $C_k$.
In fact, if not, there exist $1\le i,j \le k$ such that
$i+1<j$ and $x_i\sim x_j$.
Let $J$ be a maximal clique containing $\{x_i,x_j\}$.
Then $x_i\in H_i\cap H_{i+1}\cap J$ 
and $x_j\in H_j\cap H_{j+1}\cap J$.
In view of Lemma \ref{03lem:three maximal cliques of G with <-1/2}
we obtain $J=H_i$ or $J=H_{i+1}$ from the former condition,
and similarly $J=H_j$ or $J=H_{j+1}$ from the latter. 
In any case we come to a contradiction against that
\eqref{04eqn:smallest cycle in clique graph}
is a smallest cycle.

We next show that the cycle
$C=\langle x_1,x_2,\dots, x_k\rangle\cong C_k$ 
is isometrically embedded in $G$.
Suppose otherwise.
Then there exist $1\le i,j \le k$ such that
\begin{equation}\label{06eqn:distance comparison with cycle}
d_G(x_i,x_j)<d_{C_k}(x_i, x_j),
\end{equation}
where the right-hand side is the distance in the cycle $C_k$.
Without loss of generality, we may assume that
$1\le i<j\le k$.
Then, \eqref{06eqn:distance comparison with cycle} becomes
\begin{equation}\label{06eqn:distance comparison with cycle (1)}
d_G(x_i,x_j)<\min\{j-i, k-(j-i)\}.
\end{equation}
For $j=i+1$ we have $x_i\sim x_j$ and 
\eqref{06eqn:distance comparison with cycle (1)} does not hold.
For $j=i+2$ it follows from
\eqref{06eqn:distance comparison with cycle (1)} that
$d_G(x_i,x_j)=1$
and come to a contradiction against the argument
in the previous paragraph.
Thus, it is sufficient to derive a contradiction
from \eqref{06eqn:distance comparison with cycle (1)}
for some $j\ge i+3$.
Now take a shortest path
\[
x_i\sim y_0\sim y_1\sim \dots\sim y_s=x_j,
\qquad s=d_G(x_i,x_j).
\]
For $1\le i\le s$ take a maximal clique $J_i$ such that
$\{y_{i-1},y_i\}\subset J_i$.
Since $x_i=y_0\in H_i\cap H_{i+1}\cap J_1$,
by Lemma \ref{03lem:three maximal cliques of G with <-1/2}
we have $J_1=H_i$ or $J_1=H_{i+1}$.
Similarly, we see from
$x_j=y_s\in H_j\cap H_{j+1}\cap J_s$
that $J_s=H_j$ or $J_s=H_{j+1}$.
Thus, the path
$J_1\sim J_2\sim \dots \sim J_s$,
which is a shortest path by
Lemma \ref{03lem:constructing a chain of maximal cliques},
gives rise to an alternative path
connecting two vertices of $C$,
and hence two cyclic walks.
We will consider these two cyclic walks in details.

Consider the case where
\begin{equation}\label{04eqn:case one of four}
J_1=H_i, \quad \text{and}\quad J_s=H_j.
\end{equation}
We obtain two cyclic walks:
\begin{equation}\label{06eqn:alternative cycle 1}
H_1\sim \dots\sim H_i=J_1\sim J_2\sim\dots\sim J_s
=H_j\sim H_{j+1}\sim \dots\sim H_k\sim H_1.
\end{equation}
and 
\begin{equation}\label{06eqn:alternative cycle 2}
H_i=J_1\sim J_2\sim\dots\sim J_s=H_j \sim
H_{j-1}\sim \dots \sim H_{i+1}\sim H_i\,.
\end{equation}
The length of these walks are 
$(i-1)+(s-1)+(k-j)+1=k+s+i-j$ and
$(s-1)+(j-i)=s+j-i-1$, respectively.
In view of \eqref{06eqn:distance comparison with cycle (1)}
we consider two cases.
First, in case of $s<\min\{j-i, k-(j-i)\}=j-i$ we have
$k+s+i-j<k$, that is,
the length of \eqref{06eqn:alternative cycle 1} is less than $k$.
This contradicts to the choice of $k$.
Second, in case of $s<\min\{j-i, k-(j-i)\}=k-(j-i)$ we have
$s+j-i-1<k-1$, that is,
the length of \eqref{06eqn:alternative cycle 1} is less than $k-1$.
This contradicts to the choice of $k$.
Thus \eqref{04eqn:case one of four} does not occur.

Other than \eqref{04eqn:case one of four}
there are three more cases.
In each of these cases,
in a similar manner as in the previous case,
we may find a cycle in $\Gamma(G)$ 
which is smaller than \eqref{04eqn:smallest cycle in clique graph},
and come to a contradiction.
As a result,
the cycle $C=\langle x_1,x_2,\dots, x_k\rangle\cong C_k$ 
is isometrically embedded in $G$.

We now recall that $\mathrm{QEC}(C_k)>-1/2$ for any $k\ge4$.
In fact, the exact value of $\mathrm{QEC}(C_k)$ is known
\cite{Obata-Zakiyyah2018}.
As a consequence of (Case 2),
we obtain $\mathrm{QEC}(G)\ge \mathrm{QEC}(C_k)>-1/2$
and come to a contradiction.
Hence, $\Gamma(G)$ does not contain a smallest cycle
\eqref{04eqn:smallest cycle in clique graph} with $k\ge 4$.
This completes the proof.
\end{proof}

Summing up the above results, we state the following

\begin{theorem}\label{04thm:necessary condition}
Let $G=(V,E)$ be a graph with $\mathrm{QEC}(G)<-1/2$.
Then the clique graph $\Gamma(G)$ is a tree.
Any pair of adjacent maximal cliques $H_1$ and $H_2$ intersect
with a single vertex, i.e., $|H_1\cap H_2|=1$.
Moreover, mutually distinct three 
maximal cliques $H_1, H_2$ and $H_3$ do not intersect, 
i.e., $H_1\cap H_2\cap H_3=\emptyset$.
\end{theorem}

Thus, we say naturally that a graph $G=(V,E)$ with $\mathrm{QEC}(G)<-1/2$ 
is a block graph which admits
``cactus-like'' structure, see Figure \ref{fig:Cactus}.
On the other hand, a \textit{cactus} is defined to be a connected graph
in which no edge lies on more than one cycle.
This definition traces back to \cite{Harary-Uhlenbeck1953},
though there is ambiguity in the usage in literature.
Note that such a cactus is different from our ``cactus-like'' graph.
\begin{figure}[hbt]
\begin{center}
\includegraphics[width=220pt]{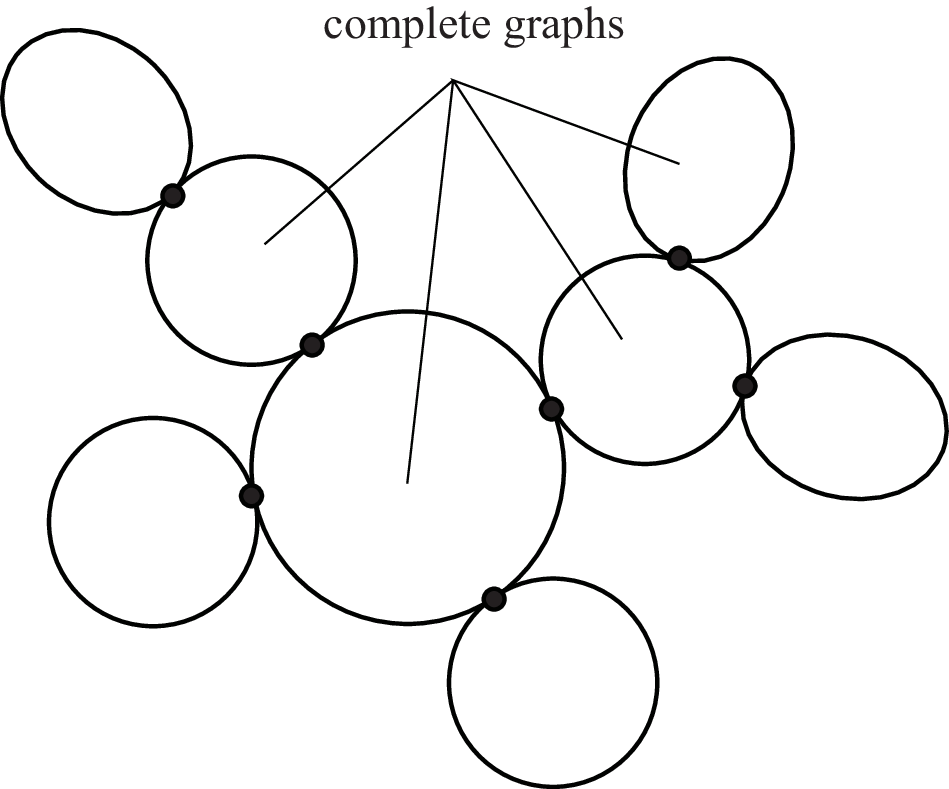}
\end{center}
\caption{``Cactus-like'' graph}
\label{fig:Cactus}
\end{figure}

\section{Graphs Consisting of Two Maximal Cliques}

For natural numbers $l\ge 1$, $m>l$ and $n>l$,
let $V=\{1,2,\dots,m+n-l\}$ and 
consider its two subsets:
\begin{equation}\label{05eqn:two cliques explicitly}
H_1=\{1,2,\dots,m\},
\qquad
H_2=\{m-l+1,m-l+2,\dots, m-l+n\}.
\end{equation}
Note that $V=H_1\cup H_2$.
Let $E$ be the set of two-element subsets 
$\{x,y\}\subset V$ satisfying $x,y\in H_1$ or
$x,y\in H_2$.
Then $G=(V,E)$ becomes a graph which is
denoted by $G=K_m\cup_l K_n$, see 
Figure \ref{fig:A connected graph consisting of two maximal cliques}.
Obviously, $G=K_m\cup_l K_n$ has exactly two maximal cliques
$H_1$ and $H_2$.
We will prove that any graph (recall that we
always assume that a graph is connected)
with exactly two maximal cliques is of this form.
\begin{figure}[hbt]
\begin{center}
\includegraphics[width=160pt]{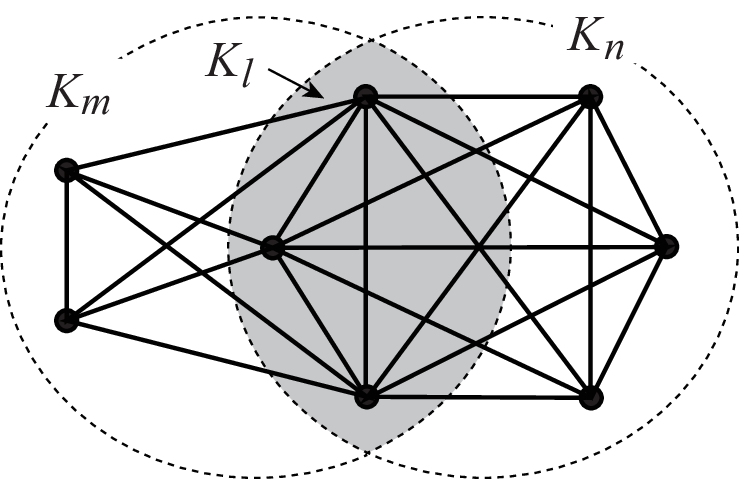}
\end{center}
\caption{A connected graph consisting of two maximal cliques}
\label{fig:A connected graph consisting of two maximal cliques}
\end{figure}

\begin{proposition}\label{05prop:two maximal cliques}
\label{05prop:A connected graph with two maximal cliques}
Let $G=(V,E)$ be a (connected) graph with exactly two maximal cliques.
Then there exist three natural numbers
$l\ge 1$, $m>l$ and $n>l$ such that $G\cong K_m\cup_l K_n$.
\end{proposition}

Although Proposition \ref{05prop:two maximal cliques}
pertains to an elementary understanding of graph theory,
for later convenience we show an outline of the argument.

\begin{lemma}\label{05lem:two maximal cliques}
Let $G=(V,E)$ be a graph with exactly two maximal cliques
$H_1$ and $H_2$.
\begin{enumerate}
\item[\upshape (1)] For any $a\in H_1\backslash H_2$
and $b\in H_2\backslash H_1$ we have $a\not\sim b$.
\item[\upshape (2)] $V=H_1\cup H_2$.
\item[\upshape (3)] $H_1\cap H_2\neq \emptyset$.
\item[\upshape (4)] For any $a\in H_1\backslash H_2$
and $b\in H_2\backslash H_1$ we have $d_G(a,b)=2$.
\end{enumerate}
\end{lemma}

\begin{proof}
(1) Suppose that $a\in H_1\backslash H_2$
and $b\in H_2\backslash H_1$ are adjacent.
There exists a maximal clique containing $\{a,b\}$,
which is different from $H_1$ and $H_2$.
This implies that $a\not\sim b$ for 
any pair of $a\in H_1\backslash H_2$ and $b\in H_2\backslash H_1$ .

(2) In order to prove by contradiction
we suppose $V \neq H_1\cup H_2$.
Take $a\in V\backslash (H_1\cup H_2)$.
Since $G$ is connected, there exists $b\in V$ such that $a\sim b$.
Then a maximal clique containing $\{a,b\}$ exists
and is different from $H_1$ and from $H_2$.
We thus come to a contradiction.

(3) Suppose that $H_1\cap H_2= \emptyset$.
Take $a\in H_1$ and $b\in H_2$ arbitrarily.
Since $G$ is connected, there exists a walk connecting
$a$ and $b$.
Since this walk is kept in $H_1\cup H_2$ by (2),
we may find $a^\prime\in H_1$ and $b^\prime\in H_2$
such that $a^\prime \sim b^\prime$.
This contradicts to the result of (1).

(4) By (1) we know that $d_G(a,b)\ge2$.
On the other hand, taking $x\in H_1\cap H_2$ we
obtain a walk $a\sim x \sim b$, which implies that
$d_G(a,b)\le2$.
\end{proof}

\begin{proof}[Proof of Proposition \ref{05prop:two maximal cliques}]
Let $H_1$ and $H_2$ be the two maximal cliques of $G$.
We set $m=|H_1|$, $n=|H_2|$ and 
$l=|H_1\cap H_2|$.
By Lemma \ref{05lem:two maximal cliques} we see that
$H_1\cong K_m$, $H_2\cong K_n$ and $H_1\cap H_2\cong K_l$
with $l\ge1$ and $m,n>l$.
Moreover, there is no edge connecting vertices  
$a\in H_1\backslash H_2$ and $b\in H_2\backslash H_1$.
We conclude that $G\cong K_m\cup_l K_n$.
\end{proof}

The distance matrix of $G=K_m\cup_l K_n$
is easily written down
according to \eqref{05eqn:two cliques explicitly}.
In fact, taking Lemma \ref{05lem:two maximal cliques} (4)
into account, we obtain the distance matrix $D$ 
in a block-matrix form as follows:
\begin{equation}\label{05eqn:distance matrix}
D=\begin{bmatrix}
J-I & J & J \\
J & J-I & 2J \\
J & 2J  & J-I 
\end{bmatrix},
\end{equation}
where $I$ is the identity matrix and
$J$ the matrix whose entries are all one
(the sizes of these matrices are understood in the context).
Then $\mathrm{QEC}(G)$ is obtained by means of the
basic formula in Proposition \ref{02prop:QEC}.
The computation is just a routine and
is deferred to the Appendix.

\begin{theorem}\label{05thm:main formula for QEC}
For $l\ge 1$, $m>l$ and $n>l$ we have
\begin{equation}\label{05eqn:main QEC}
\mathrm{QEC}(K_m\cup_l K_n)
=-1+\frac{-(m-l)(n-l)+\sqrt{\mathstrut mn(m-l)(n-l)}}{m+n-l}
\end{equation}
\end{theorem}

\begin{corollary}[{\cite[Proposition 4.4]{Baskoro-Obata2021}}]
\label{05cor:main formula for QEC with l=1}
Let $m\ge2$ and $n\ge2$.
Then $K_m\cup_1 K_n$ is a graph obtained from
$K_m$ and $K_n$ by concatenating a vertex,
in other words, it is the star product $K_m\cup_1 K_n=K_m*K_n$,
and we have
\begin{align}
\mathrm{QEC}(K_m\cup_1 K_n)
&=\mathrm{QEC}(K_m* K_n)
\nonumber \\
&=\frac{-mn+\sqrt{\mathstrut mn(m-1)(n-1)}}{m+n-1}
\nonumber \\
&=-\left(1+\sqrt{\bigg(1-\dfrac{1}{m}\bigg)
                 \bigg(1-\dfrac{1}{n}\bigg)}\,\right)^{-1}.
\label{05:eqn:formula for Km*Kn}
\end{align}
\end{corollary}

\begin{corollary}\label{05cor:main formula for QEC with l=2}
Let $m\ge3$ and $n\ge3$.
Then $K_m\cup_2 K_n$ is a graph obtained from
$K_m$ and $K_n$ by concatenating an edge
and we have
\begin{equation}\label{05eqn:main QEC with l=2}
\mathrm{QEC}(K_m\cup_2 K_n)
=\frac{-mn+m+n-2+\sqrt{\mathstrut mn(m-2)(n-2)}}{m+n-2}
\end{equation}
\end{corollary}

\begin{remark}\label{05rem:on Thm 5.3}
\normalfont
By changing parameters we obtain an alternative form of
\eqref{05eqn:main QEC} in Theorem \ref{05thm:main formula for QEC}.
For $l,m,n\ge1$ we have
\[
\mathrm{QEC}(K_{m+l}\cup_l K_{n+l})
=-1+\frac{l}{1+\sqrt{\bigg(1+\dfrac{l}{m}\bigg)\bigg(1+\dfrac{l}{n}\bigg)}}\,.
\]
This is useful to discuss estimates of $\mathrm{QEC}(K_{m+l}\cup_l K_{n+l})$.
\end{remark}

\section{Characterization of Graphs Along $QEC(P_d)$}

\begin{proposition}\label{05prop:QEC and diam Gamma(G)}
Let $d\ge3$.
If $\mathrm{QEC}(G)<\mathrm{QEC}(P_d)$,
we have
$\diam(G)\le d-2$ and $\diam(\Gamma(G))\le d-3$.
\end{proposition}

\begin{proof}
Suppose that $\diam(G)> d-2$.
Then $\diam(G)\ge d-1$ and $P_d$ is isometrically embedded in $G$.
By Proposition \ref{02prop:isometrically embedded subgraphs}
we obtain $\mathrm{QEC}(P_d)\le\mathrm{QEC}(G)$,
which contradicts to the assumption.
Therefore, if $\mathrm{QEC}(G)<\mathrm{QEC}(P_d)$,
we have $\diam(G)\le d-2$.
In that case,
since $\mathrm{QEC}(G)<\mathrm{QEC}(P_d)<-1/2$,
it follows from Theorem \ref{04thm:necessary condition}
that $\Gamma(G)$ is a tree.
We then see from Proposition \ref{03prop:Clique graph is tree}
that $\diam(\Gamma(G))=\diam(G)-1\le d-3$.
\end{proof}

\subsection{$\mathrm{QEC}(G)<\mathrm{QEC}(P_3)$}
\label{06subsec:G<P_3}

For a graph with $\mathrm{QEC}(G)<\mathrm{QEC}(P_3)$
we have $\diam(\Gamma(G))=0$,
which means that $G$ has just one maximal clique.
Hence $G=K_n$ with $n\ge2$.
Since $\mathrm{QEC}(K_n)=-1$, we have the following assertions
immediately.

\begin{proposition}[\cite{Baskoro-Obata2021}]
\label{06prop:QEC=-1}
For a graph $G$ we have $\mathrm{QEC}(G)=\mathrm{QEC}(P_2)=-1$
if and only if $G=K_n$ with $n\ge2$.
\end{proposition}

\begin{proposition}[\cite{Baskoro-Obata2021}]
There exists no graph $G$ such that
$\mathrm{QEC}(P_2)<\mathrm{QEC}(G)<\mathrm{QEC}(P_3)=-2/3$.
\end{proposition}

\subsection{$\mathrm{QEC}(G)<\mathrm{QEC}(P_4)$}
\label{06subsec:G<P_4}

Let $G$ be a graph satisfying $\mathrm{QEC}(G)<\mathrm{QEC}(P_4)$.
It follows from Proposition \ref{05prop:QEC and diam Gamma(G)}
that $\diam(\Gamma(G))\le 1$,
that is, $\diam(\Gamma(G))= 0$ or $\diam(\Gamma(G))= 1$.
The case of $\diam(\Gamma(G))= 0$ is discussed
already in Subsection \ref{06subsec:G<P_3}.

In the case of $\diam(\Gamma(G))=1$,
the clique graph $\Gamma(G)$ consists of two vertices,
which means that $G$ has exactly two maximal cliques.
By Proposition \ref{05prop:two maximal cliques}
we obtain $G=K_m\cup_l K_n$ with $l\ge1$, $m>l$ and $n>l$.
On the other hand, since $\mathrm{QEC}(G)<-1/2$,
we see from Lemma \ref{04lem:two maximal cliques of G with <-1/2}
that $l=1$.
Thus, $G$ is necessarily of the form 
$G=K_m\cup_1 K_n=K_m*K_n$ with $m\ge n\ge2$.

With the help of the formula
in Corollary \ref{05cor:main formula for QEC with l=1}
we may easily determine $m\ge n\ge2$ such that
$\mathrm{QEC}(K_m*K_n)<\mathrm{QEC}(P_4)=-(2-\sqrt2)$.
As a result we obtain the following assertions.

\begin{proposition}[\cite{Baskoro-Obata2021}]
For a graph $G$ we have 
$\mathrm{QEC}(G)=\mathrm{QEC}(P_3)=-2/3$
if and only if $G=P_3=K_2*K_2$.
\end{proposition}

\begin{proposition}[\cite{Baskoro-Obata2021}]
For a graph $G$ we have 
$\mathrm{QEC}(P_3)<\mathrm{QEC}(G)<\mathrm{QEC}(P_4)=-(2-\sqrt2)$
if and only if $G=K_m*K_2$ with $m\ge3$ or $G=K_3*K_3$.
\end{proposition}

\subsection{$\mathrm{QEC}(G)<\mathrm{QEC}(P_5)$}

If a graph $G$ satisfies $\mathrm{QEC}(G)<\mathrm{QEC}(P_5)
=-(5-\sqrt5)/5$,
then $\diam(\Gamma(G))\le 2$.
The case of $\diam(\Gamma(G))=0$ is already 
discussed in Subsection \ref{06subsec:G<P_3}.
If $\diam(\Gamma(G))= 1$,
we have $G=K_m*K_n$ with $m\ge n\ge 2$.
Then, as is discussed in Subsection \ref{06subsec:G<P_4},
we may employ the explicit formula for $\mathrm{QEC}(K_m*K_n)$ 
in Corollary \ref{05cor:main formula for QEC with l=1}.
The result will be stated in Propositions
\ref{06prop:P4=G} and \ref{06prop:diam=2 and P4<G<P5}.

Consider the case of $\diam(\Gamma(G))=2$.
Since $\Gamma(G)$ is a tree, it is 
necessarily a star $\Gamma(G)=K_{1,s}$ with $s\ge2$.
Then $G$ is a graph obtained as follows:
Let $n\ge s$ and $m_1\ge m_2\ge \dotsb \ge m_s\ge2$.
We choose $s$ vertices from $K_n$
and to each of the $s$ vertices we make 
a star product with $K_{m_1},\dots, K_{m_s}$,
see Figure \ref{fig:K_n(m_1,...m_s)}.
Such a graph is denoted by
$G=K_n*(K_{m_1},\dots, K_{m_s})$.
\begin{figure}[hbt]
\begin{center}
\includegraphics[width=180pt]{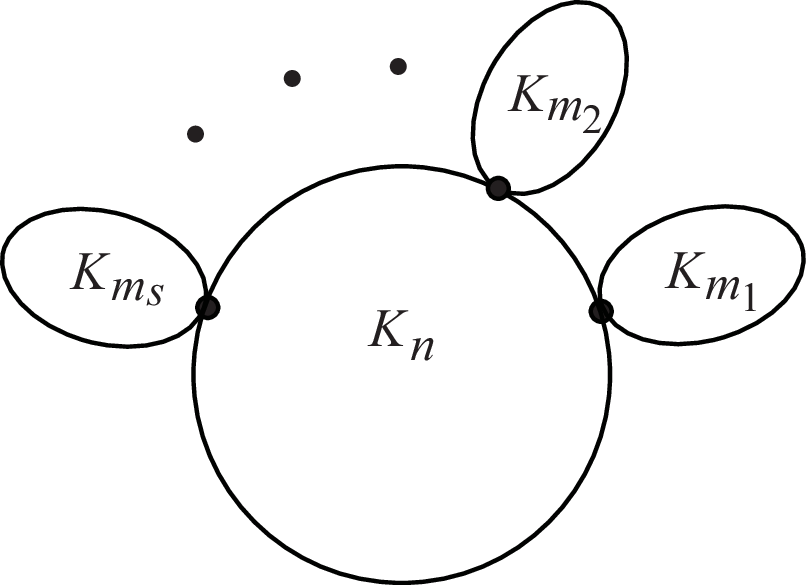}
\end{center}
\caption{$K_n*(K_{m_1},\dots, K_{m_s})$}
\label{fig:K_n(m_1,...m_s)}
\end{figure}

\begin{lemma}
Consider two graphs
$G=K_n*(K_{m_1},\dots, K_{m_s})$
with $n\ge s$, $m_1\ge m_2\ge \dotsb \ge m_s\ge2$,
and $G^\prime=K_{n^\prime}*(K_{m_1^\prime},\dots, K_{m_r^\prime})$
with $n^\prime\ge r$ and $m_1^\prime\ge m_2^\prime
\ge \dotsb \ge m_r^\prime\ge2$.
If $n^\prime\le n$, $r\le s$, $m_1^\prime\le m_1,
\dots, m_r^\prime \le m_r$,
then $G^\prime$ is isometrically embedded in $G$.
Hence $\mathrm{QEC}(G^\prime)\le \mathrm{QEC}(G)$.
\end{lemma}

\begin{proof}
Obvious.
\end{proof}

\begin{lemma}\label{06lem:P4<Kn(m1...ms)}
Let $n\ge s\ge2$ and $m_1\ge m_2\ge \dots\ge m_s\ge2$.
If $m_1\ge 3$, we have
\[
\mathrm{QEC}(P_4)<\mathrm{QEC}(K_n*(K_{m_1},\dots, K_{m_s})).
\]
\end{lemma}

\begin{proof}
It is known that
\[
\mathrm{QEC}(K_2*(K_3,K_2))
=-\frac{2(6-\sqrt{21})}{5}\approx -0.5669,
\]
see \cite{Obata-Zakiyyah2018},
where $K_2*(K_3,K_2)$ is referred to as No.5-7.
Then we have
\[
\mathrm{QEC}(P_4)
<\mathrm{QEC}(K_2*(K_3,K_2))
<\mathrm{QEC}(P_5).
\]
On the other hand, for $n\ge s\ge2$ and $m_1\ge3$,
$K_n*(K_{m_1},\dots, K_{m_s})$ contains
$K_2*(K_3,K_2)$ as an isometrically embedded subgraph.
Hence
\[
\mathrm{QEC}(P_4)
<\mathrm{QEC}(K_2*(K_3,K_2))
\le\mathrm{QEC}(K_n*(K_{m_1},\dots, K_{m_s})),
\]
as desired.
\end{proof}

\begin{lemma}[{\cite[Theorem 4.3]{Baskoro-Obata2021}}]
\label{06lem:QEC(Kn(K2...K2)}
For $n\ge s\ge2$ we have
\[
\mathrm{QEC}(K_n*(\overbrace{K_2,\dots,K_2}^{\text{$s$ times}}))
=-(2-\sqrt2)=\mathrm{QEC}(P_4).
\]
\end{lemma}

\begin{proposition}\label{06prop:P4=G}
For a graph $G$ we have 
$\mathrm{QEC}(G)=\mathrm{QEC}(P_4)=-(2-\sqrt2)$
if and only if $G=K_4*K_3$ or
$G=K_n*(K_2,\dots,K_2)$ ($K_2$ appears $s$ times)
with $n\ge s\ge 2$.
\end{proposition}

\begin{proof}
As is discussed in Subsection \ref{06subsec:G<P_4},
a graph $G$ with $\diam(\Gamma(G))=1$ is of the form
$K_m*K_n$ with $m\ge n\ge2$.
Then, using the explicit formula for $\mathrm{QEC}(K_m*K_n)$ 
in Corollary \ref{05cor:main formula for QEC with l=1},
we see easily that
$\mathrm{QEC}(K_m*K_n)=\mathrm{QEC}(P_4)$ if and only if
$m=4$ and $n=3$.

A graph $G$ with $\diam(\Gamma(G))=2$ is of the form
$G=K_n*(K_{m_1},\dots, K_{m_s})$ with $n\ge s\ge2$
and $m_1\ge m_2\ge\dotsb\ge m_s\ge2$.
By Lemma \ref{06lem:P4<Kn(m1...ms)},
$\mathrm{QEC}(G)\le \mathrm{QEC}(P_4)$ may occur 
only when $m_1=m_2=\dotsb= m_s=2$.
On the other hand, in that case,
the equality $\mathrm{QEC}(G)=\mathrm{QEC}(P_4)$ holds
by Lemma \ref{06lem:QEC(Kn(K2...K2)}.
\end{proof}

\begin{proposition}\label{06prop:diam=2 and P4<G<P5}
For $m\ge n\ge 2$ we have
$\mathrm{QEC}(P_4)<\mathrm{QEC}(K_m*K_n)<\mathrm{QEC}(P_5)$
if and only if
\begin{enumerate}
\item[\upshape (i)] $n=3$ and $5\le m\le 54$;
\item[\upshape (ii)] $n=4$ and $4\le m\le 7$;
\item[\upshape (iii)] $n=m=5$.
\end{enumerate}
\end{proposition}

\begin{proof}
Straightforward by
the explicit formula for $\mathrm{QEC}(K_m*K_n)$.
\end{proof}

By Proposition \ref{06prop:diam=2 and P4<G<P5} 
all graphs $G$ such that
$\mathrm{QEC}(P_4)<\mathrm{QEC}(G)<\mathrm{QEC}(P_5)$
with $\diam(\Gamma(G))=1$ are determined.
The case of $\diam(\Gamma(G))=2$,
i.e., $G=K_n*(K_{m_1},\dots, K_{m_s})$ remains to be checked.
The work in this line is in progress.

\setcounter{section}{1}
\section*{Appendix: Calculating $\mathrm{QEC}(K_m\cup_l K_n)$}
\setcounter{equation}{0}
\setcounter{theorem}{0}
\renewcommand{\thesection}{\Alph{section}}
\renewcommand{\theequation}{\Alph{section}.\arabic{equation}}

Let $D$ be the distance matrix of $G=K_m\cup_l K_n$,
where $l\ge1$ and $m,n>l$.
Using the block-matrix form of $D$ as in 
\eqref{05eqn:distance matrix},
we will calculate $\mathrm{QEC}(G)$ explicitly.

For $f\in \mathbb{R}^l$, $g\in \mathbb{R}^{m-l}$
and $h\in \mathbb{R}^{n-l}$ we set
\begin{align*}
\psi(f,g,h)
&=\left\langle 
\begin{bmatrix}
f \\ g \\ h
\end{bmatrix},
D\begin{bmatrix}
f \\ g \\ h
\end{bmatrix}
\right\rangle \\
&=\langle \1,f\rangle^2+\langle \1,g\rangle^2+\langle \1,h\rangle^2
  -\langle f,f\rangle^2-\langle g,g\rangle^2-\langle h,h\rangle^2 \\
& \qquad
  +2\langle \1,f\rangle\langle \1,g\rangle
 +2\langle \1,f\rangle\langle \1,h\rangle
 +4\langle \1,g\rangle\langle \1,h\rangle
\end{align*}
and 
\begin{align}
\varphi(f,g,h,\lambda,\mu) 
=\psi(f,g,h)
&-\lambda(\langle f,f\rangle +\langle g,g\rangle +\langle h,h\rangle-1)
\nonumber\\
&-\mu(\langle \1,f\rangle +\langle \1,g\rangle +\langle \1,h\rangle).
\label{aeqn:starting phi}
\end{align}
It then follows from Proposition \ref{02prop:QEC} that
$\mathrm{QEC}(G)$ coincides with the maximum of
$\lambda\in\mathbb{R}$ appearing 
in the stationary points of $\psi(f,g,h,\lambda,\mu)$.
Since $G$ is not complete, it is sufficient to
explore stationary points of $\psi(f,g,h,\lambda,\mu)$
with $\lambda>-1$.

By direct computation together with condition 
\[
\langle \bm{1},f\rangle
+\langle \bm{1},g\rangle
+\langle \bm{1},h\rangle=0
\]
we have
\begin{equation}\label{aeqn:stationary eqn 1}
\frac{\partial \varphi}{\partial f_i}
=-2(\lambda+1)f_i-\mu
=0.
\end{equation}
Then we see that $f_i$ is constant independent of $1\le i\le l$,
say $f_i=\xi$.
Thus, \eqref{aeqn:stationary eqn 1} becomes
\begin{equation}\label{aeqn:for xi}
\xi=-\frac{1}{\lambda+1}\cdot\frac{\mu}{2}\,.
\end{equation}
(From the beginning we may assume that $\lambda>-1$
as noted before.)
Similarly, it follows from
$\partial \varphi/\partial g_i
=\partial \varphi/\partial h_i=0$
that $g_i$ and $h_i$ are respectively constant.
Setting $g_i=\eta$ and $h_i=\zeta$, we obtain
\begin{align}
(\lambda+1)\eta-(n-l)\zeta=-\frac{\mu}{2}\,, 
\label{aeqn:for eta and zeta 11}\\
-(m-l)\eta+(\lambda+1)\zeta=-\frac{\mu}{2}\,,
\label{aeqn:for eta and zeta 12}
\end{align}
and the constraints become
\begin{align}
l\xi+(m-l)\eta+(n-l)\zeta=0, 
\label{aeqn:condition 1} \\
l\xi^2+(m-l)\eta^2+(n-l)\zeta^2=1.
\label{aeqn:condition 2}
\end{align}
Our task is to solve the system of equations
\eqref{aeqn:for xi}--\eqref{aeqn:condition 2}.

In view of
\eqref{aeqn:for eta and zeta 11}
and \eqref{aeqn:for eta and zeta 12}
we set
\begin{equation}\label{aeqn:delta}
\Delta=\det\begin{bmatrix}
\lambda+1 & -(n-l) \\
-(m-l) & \lambda+1
\end{bmatrix}
=(\lambda+1)^2-(m-l)(n-l).
\end{equation}

(Case I) $\Delta\neq0$.
The equations \eqref{aeqn:for eta and zeta 11}
and \eqref{aeqn:for eta and zeta 12}
have a unique solution:
\begin{equation}\label{aeqn:eta and zeta}
\eta=-\frac{1}{\lambda+n+1-l}\cdot\frac{\mu}{2}\,,
\qquad
\zeta=-\frac{1}{\lambda+m+1-l}\cdot\frac{\mu}{2}\,.
\end{equation}
Inserting \eqref{aeqn:for xi} and \eqref{aeqn:eta and zeta}
into \eqref{aeqn:condition 1}, 
we obtain
\begin{equation}\label{aeqn:equation for lambda}
l\Delta+(\lambda+1)
\{(m-l)(\lambda+n+1-l)+(n-l)(\lambda+m+1-l)\}=0
\end{equation}
after simple calculation. 
The solutions are easily written down as
\[
\lambda_\pm=-1+
\frac{-(m-l)(n-l)\pm \sqrt{mn(m-l)(n-l)}}{m+n-l}\,.
\]
Checking that $\lambda_+\neq -1$ and $\Delta\neq0$
for $\lambda=\lambda_+$,
we see that $\lambda_+$ is a candidate of
$\mathrm{QEC}(G)$.

(Case II) $\Delta=0$.
From \eqref{aeqn:for eta and zeta 11}
and \eqref{aeqn:for eta and zeta 12} we obtain
\[
(\lambda+n-l+1)\frac{\mu}{2}=0,
\qquad
(\lambda+m-l+1)\frac{\mu}{2}=0.
\]
If $m\neq n$, we obtain $\mu=0$ and $\xi=\eta=\zeta=0$, which
do not fulfill \eqref{aeqn:condition 2}.
Hence there is no stationary points.
Assume that $m=n$.
Then $\lambda=-1-(m-l)$ appears in the stationary points.
However, since we are only interested in $\lambda>-1$,
there exists no candidate for our 
$\mathrm{QEC}(G)$ in the case of $\Delta=0$.

Finally, we conclude from (Case I) and (Case II) that
$\lambda_+=\mathrm{QEC}(G)$.

\section*{Acknowledgements} 
NO is grateful to Institut Teknologi Bandung (ITB)
for their hospitality as an adjunct professor during 2023--2024.
This work is also supported in part by
JSPS Grant-in-Aid for Scientific Research No.~23K03126.
The authors thank the referees for their valuable comments 
which improved the presentation of this paper,
in particular, for the useful expression in Remark \ref{05rem:on Thm 5.3}
and the recent paper \cite{Guo-Zhou2023}.


\end{document}